\documentclass[a4paper,11pt]{article}
\usepackage{graphicx}
\usepackage{multirow}
\usepackage{color}
\usepackage{latexsym}
\usepackage{amssymb}
\usepackage{amsfonts}
\usepackage{amsmath}
\usepackage{amsthm}
\usepackage{indentfirst}
\usepackage{mathrsfs}
\usepackage{tikz}
\usetikzlibrary{patterns}
\usetikzlibrary{hobby,decorations.pathreplacing,arrows.meta}

\usepackage[latin1]{inputenc}

\usepackage{fullpage}

\newtheorem{proposition}{Proposition}
\newtheorem{lemma}{Lemma}
\newtheorem{conjecture}{Conjecture}

\date{}

\author{Antonio Bernini
\thanks{Dipartimento di Matematica e Informatica ``U. Dini'',
University of Firenze, Firenze, Italy. {
\tt\  antonio.bernini@unifi.it, luca.ferrari@unifi.it}}
\and Matteo Cervetti
\thanks{Dipartimento di Matematica, University of Trento, Trento, Italy.
{\tt \ matteo.cervetti@unitn.it}}
\and Luca Ferrari $^*$
\and Einar Steingr\'imsson $^{*,\!\!}$
\thanks{Department of Computer and Information Sciences,
University of Strathclyde, Glasgow, Scotland, UK. {\tt \  einar@alum.mit.edu}  }}

\title{Enumerative combinatorics of intervals in the Dyck pattern poset}

\begin{document}

\maketitle

\begin{abstract}
We initiate the study of the enumerative combinatorics of the intervals in the Dyck pattern poset.
More specifically, we find some closed formulas to express the size of some specific intervals, as well as the number of their covering relations.
In most of the cases, we are also able to refine our formulas by rank.
We also provide the first results on the M\"obius function of the Dyck pattern poset,
giving for instance a closed expression for the M\"obius function of initial intervals whose maximum is a Dyck path having exactly two peaks.
\end{abstract}


\section{Introduction}

\noindent
The \emph{Dyck pattern poset} was first introduced in \cite{BFPW} and further studied in \cite{BBFGPW}.
A \emph{Dyck path} is a lattice path starting from the origin of a fixed Cartesian coordinate system,
ending on the $x$-axis, never falling below the $x$-axis and using only two types of steps,
namely up steps $U=(1,1)$ and down steps $D=(1,-1)$.
The sequence of up and down steps of a Dyck path is a word on the alphabet $\{ U,D\}$ such that
each prefix has at least as many $U$'s as $D$'s and the total number of $U$'s and $D$'s is the same.
Such words are commonly called \emph{Dyck words}.
The total number of letters/steps of a Dyck word/path is called the \emph{length} of the word/path.
In the following we will frequently switch between paths and words,
and in particular we will use the same notations when no confusion is likely to arise.
Given two Dyck paths $P,Q$, we say that $P\leq Q$ when $P$ is a subword of $Q$
(i.e. there exists a subsequence of the letters of $Q$ which, read from left to right, are equal to $P$).
In this case, we also say that $P$ is a \emph{pattern} of $Q$,
and any subword of $Q$ which is equal to $P$ is called an \emph{occurrence} of $P$ in $Q$.
So, for instance, $UUDD\leq UDUDUD$, whereas $UUDDUD$ and $UUDUUUDDDD$ are incomparable.
The Dyck pattern poset has a minimum, which is the path $UD$, and has no maximum;
moreover, it is graded, the rank of an element being its semilength.

In the above mentioned papers some enumerative properties of the Dyck pattern poset have been investigated,
mainly focusing on pattern avoidance questions.
Here we start the analysis of the enumerative combinatorics of the intervals of this poset.

\bigskip

Given any poset, one of the most natural aspects to investigate is the structure of its intervals.
This has been done in several combinatorially interesting posets, such as for Tamari lattices \cite{CCP,F},
the Bruhat order \cite{T}, the consecutive pattern poset \cite{EM}, to cite just a few.
From this point of view, a fairly general problem is that of counting (saturated) chains
(here ``saturated" means that the chain cannot be extended except at the beginning and at the end).
Special instances of this problem are the enumeration of the elements and of the covering relations of the interval
(which are saturated chains of length 0 and 1, respectively).
Another important quantity associated to a (combinatorially interesting) poset is
the M\"obius function of its intervals.
For our purposes, we can define the M\"obius function
$\mu:\mathcal{P}^2 \rightarrow \mathbb{Z}$ of the poset $\mathcal{P}$ in the following recursive way
(for $x\leq y$):

$$
\left\{ \begin{array}{ll}
\mu (x,x)=1,\\
\mu (x,y)=-\sum_{x\leq z<y}\mu (x,z), \qquad \textnormal{when $x<y$.}
\end{array}
\right. .
$$

In the present paper we analyze a few types of \emph{initial intervals} $[UD,P]$ in the Dyck pattern poset.
More specifically, we first consider the case in which $P=(UD)^n$, for some $n\in \mathbb{N}$,
for which we are able to determine the cardinality of the interval, also refined by rank.
Then we examine in detail the case in which $P$ has exactly two peaks
(a peak of a Dyck path $P$ is an occurrence of the Dyck path $UD$ as a consecutive pattern in $P$):
here we find formulas both for the cardinality of the interval
(and also in this case we have a refined version for ranks) and for the number of covering relations.
We find also a nice bijection between
Dyck paths inside the interval $[UD,P]$ having two peaks and squares fitting inside a rectangle of appropriate dimensions.
Finally, we give also a complete description of the M\"obius function of such intervals.
We remark that
the computation of the M\"obius function of the Dyck pattern poset is still open for general intervals,
and the results contained in the present paper are the first ones for this poset.
In the last section, together with some proposals for further work,
we also provide some additional results and conjectures on the M\"obius function which suggest that
the Dyck pattern intervals have a nice structure that certainly deserves to be better investigated.

\bigskip

We close this Introduction by fixing the main notations we are using throughout the paper.

Given a poset $\mathcal{P}$ and a nonnegative integer $\ell$,
a $\textit{saturated chain of length}$ $\ell$ in $\mathcal{P}$ is
a sequence $(x_{0},x_{1},...,x_{\ell})$ of $\ell +1$ elements of $\mathcal{P}$ such that
$x_{0}\prec x_{1}\prec ...\prec x_{\ell}$, where $\prec$ denotes the covering relation of $\mathcal{P}$.
For a finite poset $\mathcal{P}$,
denote with $s_{\ell}(\mathcal{P})$ the number of saturated chains of length $\ell$ in $\mathcal{P}$.
In particular, $s_{0}(\mathcal{P})$ is the number of elements of $\mathcal{P}$,
and $s_{1}(\mathcal{P})$ is the number of edges of the Hasse diagram of $\mathcal{P}$ (which is also the number of coverings relations in $\mathcal P$).
When $\mathcal{P}$ is graded and $\ell ,k\in \mathbb{N}$,
the number of saturated chains of length $\ell$ whose top element has rank $k$ will be denoted $s_{\ell}^{(k)}(\mathcal{P})$.
Therefore $s_{\ell}(\mathcal{P})=\sum_{k\geq \ell}s_{\ell}^{(k)}(\mathcal{P})$.
In particular, $s_0 ^{(k)}(\mathcal{P})$ is the number of elements of $\mathcal{P}$ having rank $k$.

Given $x\in \mathcal{P}$, we write $\Delta(x)$ for the number of elements in $\mathcal{P}$ covered by $x$.
Moreover, $\Delta_{t}(\mathcal{P})$ will denote the number of $x\in \mathcal{P}$ such that $\Delta(x)=t$.
As a consequence, we have that $s_{1}(\mathcal{P})=\sum_{t\geq 0}t\cdot \Delta_{t}(\mathcal{P})$.

\section{The interval $[UD,(UD)^{n}]$}

\subsection{Size of the interval}

Our first goal is to find an explicit formula for the number of elements of the interval $[UD,(UD)^{n}]$,
for $n\in \mathbb{N}$.
Recall that, given $n,k\in \mathbb{N}$, the $\textit{Narayana number}$ $N_{n,k}$ is defined as
the number of Dyck paths of semilength $n$ having $k$ peaks.
It is well known that $N_{0,0}=1$, $N_{n,k}=\frac{1}{n}\binom{n}{k}\binom{n}{k-1}$ for $n,k\geq 1$
and $N_{n,k}=0$ in the remaining cases.
%

Suppose now $P$ is a Dyck path and denote with $\textsf{asc}(P)$ the number of ascents of $P$,
where an \emph{ascent} of a Dyck path is a maximal consecutive substrings of $P$ of the form $U^{m}$,
for some $m>0$.
It is clear that $\textsf{asc}(P)$ also counts the number of peaks of $P$,
in particular $N_{n,k}$ also counts the number of Dyck paths of semilength $n$ with $k$ ascents.
The next lemma characterizes Dyck paths in the interval $[UD,(UD)^{n}]$ in terms of the number of ascents.

\begin{lemma} Let $n>0$, $k\in \{1,...,n\}$ and $P$ be a Dyck path of semilength $k$.
Then $P\leq (UD)^{n}$ if and only if $\textsf{asc}(P)\geq 2k-n$.
\end{lemma}

In order to prove this lemma,
it is convenient to regard it as a special case of the following slight generalization.

\begin{lemma} For any positive integers $n,m,\alpha_1 ,\ldots ,\alpha_m ,\beta_1 ,\ldots ,\beta_m$,
set $\alpha=\sum_{i=1}^{m}\alpha_i$ and $\beta=\sum_{i=1}^{m}\beta_i$.
Then the string $U^{\alpha_{1}}D^{\beta_{1}}U^{\alpha_{2}}D^{\beta_{2}}\cdots U^{\alpha_{m}}D^{\beta_{m}}$
is a substring of $(UD)^{n}$ if and only if $\alpha +\beta -n\leq m\leq n$.
\end{lemma}

\begin{proof}
Set $P=U^{\alpha_{1}}D^{\beta_{1}}U^{\alpha_{2}}D^{\beta_{2}}...U^{\alpha_{m}}D^{\beta_{m}}$.

Suppose first that $P$ is a substring of $(UD)^n$.
Then, for each step of $P$ not belonging to a peak,
we need one factor $UD$ from $(UD)^n$ in order to embed $P$ into $(UD)^n$.
Moreover, each peak of $P$ needs just one factor from $(UD)^n$.
Thus the total number of factors $UD$ of $(UD)^n$ must be at least the sum of the two above quantities,
that is
$$
n\geq \left( \sum_{i=1}^{m}(\alpha_i -1)+\sum_{i=1}^{m}(\beta_i -1)\right) +m=\alpha +\beta -m,
$$
which implies the desired inequality.



Suppose now that $n\geq \alpha+\beta-m$. We look for an occurrence of $P$ in $(UD)^n$.
It is not hard to see that each step of $P$ not belonging to a peak, as well as each peak of $P$,
requires precisely one factor $UD$ from $(UD)^n$.
Thus, for any $i$, in order to embed $U^{\alpha_i}D^{\beta_i}=U^{\alpha_i -1}(UD)D^{\beta_i-1}$ into $(UD)^n$
we need $(\alpha_i -1)+1+(\beta_i -1)=\alpha_i +\beta_i -1$ factors $UD$.
Therefore, we can embed $P$ into $(UD)^n$ provided that $n$ is at least
$\sum_{i=1}^{m}(\alpha_i +\beta_i -1)=\alpha +\beta -m$, which is the hypothesis.

%

\end{proof}

Now as an immediate consequence of this lemma we can deduce an explicit formula for $s_0^{(k)}([UD,(UD)^{n}])$ and $s_0([UD,(UD)^{n}])$ when $n\in \mathbb{N}$ and $k\in \{1,...,n\}$.

\begin{proposition} Let $n>0$ and $k\in \{1,...,n\}$, then
\begin{itemize}
\item[(i)]$$s_0^{(k)}([UD,(UD)^{n}])=\sum_{m=\max\{1,2k-n\}}^{k}N_{k,m}$$
\item[(ii)] \begin{equation}\label{UD}
 s_{0}([UD,(UD)^{n}])=\sum_{k=1}^{n}\sum_{m=\max\{1,2k-n\}}^{k}N_{k,m}.
 \end{equation}
\end{itemize}
\end{proposition}

\bigskip

\noindent
In Figure \ref{fig1} the Hasse diagram of the interval
$[UD,(UD)^5]$ is depicted.

\begin{figure}[h]
	\begin{center}
		{\begin{tikzpicture}
			\node (1) at (0,0)
			{\scalebox{0.2}
				{\begin{tikzpicture}
					\draw  (0,0)  --  (1,1) -- (2,0);
					\draw[fill] (0,0) circle[radius=0.1];
					\draw[fill] (1,1) circle[radius=0.1];
					\draw[fill] (2,0) circle[radius=0.1];
					\end{tikzpicture}}};
			\node (2) at (-1,1)
			{\scalebox{0.2}
				{\begin{tikzpicture}
					\draw (0,0) --  (1,1) -- (2,0) -- (3,1) -- (4,0);
					\draw[fill] (0,0) circle[radius=0.1];
					\draw[fill] (1,1) circle[radius=0.1];
					\draw[fill] (2,0) circle[radius=0.1];
					\draw[fill] (3,1) circle[radius=0.1];
					\draw[fill] (4,0) circle[radius=0.1];
					\end{tikzpicture}}};
			\node (3) at (1,1)
			{\scalebox{0.2}
				{\begin{tikzpicture}
					\draw  (0,0) --  (2,2) --  (4,0);
					\draw[fill] (0,0) circle[radius=0.1];
					\draw[fill] (1,1) circle[radius=0.1];
					\draw[fill] (2,2) circle[radius=0.1];
					\draw[fill] (3,1) circle[radius=0.1];
					\draw[fill] (4,0) circle[radius=0.1];
					\end{tikzpicture}}};
			\node (4) at (-4,2.5)
			{\scalebox{0.2}
				{\begin{tikzpicture}
					\draw  (0,0) --  (1,1) --  (2,0) -- (3,1) -- (4,0) -- (5,1) -- (6,0);
					\draw[fill] (0,0) circle[radius=0.1];
					\draw[fill] (1,1) circle[radius=0.1];
					\draw[fill] (2,0) circle[radius=0.1];
					\draw[fill] (3,1) circle[radius=0.1];
					\draw[fill] (4,0) circle[radius=0.1];
					\draw[fill] (5,1) circle[radius=0.1];
					\draw[fill] (6,0) circle[radius=0.1];
					\end{tikzpicture}}};
			\node (5) at (-2,2.5)
			{\scalebox{0.2}
				{\begin{tikzpicture}
					\draw  (0,0) --  (2,2) --  (4,0) -- (5,1) --
					(6,0);
					\draw[fill] (0,0) circle[radius=0.1];
					\draw[fill] (1,1) circle[radius=0.1];
					\draw[fill] (2,2) circle[radius=0.1];
					\draw[fill] (3,1) circle[radius=0.1];
					\draw[fill] (4,0) circle[radius=0.1];
					\draw[fill] (5,1) circle[radius=0.1];
					\draw[fill] (6,0) circle[radius=0.1];
					\end{tikzpicture}}};
			\node (6) at (0,2.5)
			{\scalebox{0.2}
				{\begin{tikzpicture}
					\draw  (0,0) --  (1,1) --  (2,0) -- (4,2) --
					(6,0);
					\draw[fill] (0,0) circle[radius=0.1];
					\draw[fill] (1,1) circle[radius=0.1];
					\draw[fill] (2,0) circle[radius=0.1];
					\draw[fill] (3,1) circle[radius=0.1];
					\draw[fill] (4,2) circle[radius=0.1];
					\draw[fill] (5,1) circle[radius=0.1];
					\draw[fill] (6,0) circle[radius=0.1];			
					\end{tikzpicture}}};
			\node (7) at (2,2.5)
			{\scalebox{0.2}
				{\begin{tikzpicture}
					\draw  (0,0) --  (2,2) --  (3,1) -- (4,2) --
					(6,0);
					\draw[fill] (0,0) circle[radius=0.1];
					\draw[fill] (1,1) circle[radius=0.1];
					\draw[fill] (2,2) circle[radius=0.1];
					\draw[fill] (3,1) circle[radius=0.1];
					\draw[fill] (4,2) circle[radius=0.1];
					\draw[fill] (5,1) circle[radius=0.1];
					\draw[fill] (6,0) circle[radius=0.1];
					\end{tikzpicture}}};
			\node (8) at (4,2.5)
			{\scalebox{0.2}
				{\begin{tikzpicture}
					\draw  (0,0) --  (3,3) -- (6,0);
					\draw[fill] (0,0) circle[radius=0.1];
					\draw[fill] (1,1) circle[radius=0.1];
					\draw[fill] (2,2) circle[radius=0.1];
					\draw[fill] (3,3) circle[radius=0.1];
					\draw[fill] (4,2) circle[radius=0.1];
					\draw[fill] (5,1) circle[radius=0.1];
					\draw[fill] (6,0) circle[radius=0.1];
					\end{tikzpicture}}};
			\node (9) at (-6,5)
			{\scalebox{0.2}
				{\begin{tikzpicture}
					\draw  (0,0) --  (1,1) --  (2,0) -- (3,1) -- (4,0)
					-- (5,1) -- (6,0) -- (7,1) -- (8,0);
					\draw[fill] (0,0) circle[radius=0.1];
					\draw[fill] (1,1) circle[radius=0.1];
					\draw[fill] (2,0) circle[radius=0.1];
					\draw[fill] (3,1) circle[radius=0.1];
					\draw[fill] (4,0) circle[radius=0.1];
					\draw[fill] (5,1) circle[radius=0.1];
					\draw[fill] (6,0) circle[radius=0.1];
					\draw[fill] (7,1) circle[radius=0.1];
					\draw[fill] (8,0) circle[radius=0.1];
					\end{tikzpicture}}};
			\node (10) at (-4,5)
			{\scalebox{0.2}
				{\begin{tikzpicture}
					\draw  (0,0) --  (2,2) --  (4,0) -- (5,1) -- (6,0)
					-- (7,1) -- (8,0);
					\draw[fill] (0,0) circle[radius=0.1];
					\draw[fill] (1,1) circle[radius=0.1];
					\draw[fill] (2,2) circle[radius=0.1];
					\draw[fill] (3,1) circle[radius=0.1];
					\draw[fill] (4,0) circle[radius=0.1];
					\draw[fill] (5,1) circle[radius=0.1];
					\draw[fill] (6,0) circle[radius=0.1];
					\draw[fill] (7,1) circle[radius=0.1];
					\draw[fill] (8,0) circle[radius=0.1];
					\end{tikzpicture}}};
			\node (11) at (-2,5)
			{\scalebox{0.2}
				{\begin{tikzpicture}
					\draw  (0,0) --  (2,2) --  (3,1) -- (4,2) -- (6,0)
					-- (7,1) -- (8,0);
					\draw[fill] (0,0) circle[radius=0.1];
					\draw[fill] (1,1) circle[radius=0.1];
					\draw[fill] (2,2) circle[radius=0.1];
					\draw[fill] (3,1) circle[radius=0.1];
					\draw[fill] (4,2) circle[radius=0.1];
					\draw[fill] (5,1) circle[radius=0.1];
					\draw[fill] (6,0) circle[radius=0.1];
					\draw[fill] (7,1) circle[radius=0.1];
					\draw[fill] (8,0) circle[radius=0.1];
					\end{tikzpicture}}};
			\node (12) at (0,5)
			{\scalebox{0.2}
				{\begin{tikzpicture}
					\draw  (0,0) --  (2,2) --  (3,1) -- (4,2) -- (5,1)
					-- (6,2) -- (8,0);
					\draw[fill] (0,0) circle[radius=0.1];
					\draw[fill] (1,1) circle[radius=0.1];
					\draw[fill] (2,2) circle[radius=0.1];
					\draw[fill] (3,1) circle[radius=0.1];
					\draw[fill] (4,2) circle[radius=0.1];
					\draw[fill] (5,1) circle[radius=0.1];
					\draw[fill] (6,2) circle[radius=0.1];
					\draw[fill] (7,1) circle[radius=0.1];
					\draw[fill] (8,0) circle[radius=0.1];
					\end{tikzpicture}}};
			\node (13) at (2,5)
			{\scalebox{0.2}
				{\begin{tikzpicture}
					\draw  (0,0) --  (1,1) --  (2,0) -- (4,2) -- (6,0)
					-- (7,1) -- (8,0);
					\draw[fill] (0,0) circle[radius=0.1];
					\draw[fill] (1,1) circle[radius=0.1];
					\draw[fill] (2,0) circle[radius=0.1];
					\draw[fill] (3,1) circle[radius=0.1];
					\draw[fill] (4,2) circle[radius=0.1];
					\draw[fill] (5,1) circle[radius=0.1];
					\draw[fill] (6,0) circle[radius=0.1];
					\draw[fill] (7,1) circle[radius=0.1];
					\draw[fill] (8,0) circle[radius=0.1];
					\end{tikzpicture}}};
			\node (14) at (4,5)
			{\scalebox{0.2}
				{\begin{tikzpicture}
					\draw  (0,0) --  (1,1) --  (2,0) -- (4,2) -- (5,1)
					-- (6,2) -- (8,0);
					\draw[fill] (0,0) circle[radius=0.1];
					\draw[fill] (1,1) circle[radius=0.1];
					\draw[fill] (2,0) circle[radius=0.1];
					\draw[fill] (3,1) circle[radius=0.1];
					\draw[fill] (4,2) circle[radius=0.1];
					\draw[fill] (5,1) circle[radius=0.1];
					\draw[fill] (6,2) circle[radius=0.1];
					\draw[fill] (7,1) circle[radius=0.1];
					\draw[fill] (8,0) circle[radius=0.1];
					\end{tikzpicture}}};
			\node (15) at (6,5)
			{\scalebox{0.2}
				{\begin{tikzpicture}
					\draw  (0,0) --  (1,1) --  (2,0) -- (3,1) -- (4,0)
					-- (6,2) -- (8,0);
					\draw[fill] (0,0) circle[radius=0.1];
					\draw[fill] (1,1) circle[radius=0.1];
					\draw[fill] (2,0) circle[radius=0.1];
					\draw[fill] (3,1) circle[radius=0.1];
					\draw[fill] (4,0) circle[radius=0.1];
					\draw[fill] (5,1) circle[radius=0.1];
					\draw[fill] (6,2) circle[radius=0.1];
					\draw[fill] (7,1) circle[radius=0.1];
					\draw[fill] (8,0) circle[radius=0.1];	
					\end{tikzpicture}}};
			\node (16) at (0,6.5)
			{\scalebox{0.2}
				{\begin{tikzpicture}
					\draw  (0,0) --  (1,1) --  (2,0) -- (3,1) -- (4,0)
					-- (5,1) -- (6,0) -- (7,1) -- (8,0) -- (9,1) --
					(10,0);
					\draw[fill] (0,0) circle[radius=0.1];
					\draw[fill] (1,1) circle[radius=0.1];
					\draw[fill] (2,0) circle[radius=0.1];
					\draw[fill] (3,1) circle[radius=0.1];
					\draw[fill] (4,0) circle[radius=0.1];
					\draw[fill] (5,1) circle[radius=0.1];
					\draw[fill] (6,0) circle[radius=0.1];
					\draw[fill] (7,1) circle[radius=0.1];
					\draw[fill] (8,0) circle[radius=0.1];
					\draw[fill] (9,1) circle[radius=0.1];
					\draw[fill](10,0)circle[radius=0.1];
					\end{tikzpicture}}};
			\draw[very thin] (1) -- (2) (1) -- (3) (2) -- (4) (2)--(5) (2)--(6) (2)--(7) (3)--(4) (3)--(5) (3)--(6) (3)--(7) (3)--(8) (4)--(9) (5)--(9) (6)--(9) (7)--(9) (10) -- (4) (10)--(5) (10)--(7) (12)--(4) (11)--(4) (11)--(5) (11)--(7) (11)--(8) (12)--(4) (12)--(5) (12)--(6) (12)--(7) (12)--(8) (13)--(4) (13)--(5) (13)--(6) (13)--(8)  (14)--(4) (14)--(6) (14)--(7) (14)--(8)  (15)--(4) (15)--(6) (15)--(8)  (16) -- (9) (16)--(10) (16)--(11) (16)--(12) (16)--(13) (16)--(14) (16)--(15);
	\end{tikzpicture}}
\end{center}
\caption{The Hasse diagram of the interval $[UD,(UD)^{5}]$ in the
Dyck pattern poset}
\label{fig1}
\end{figure}
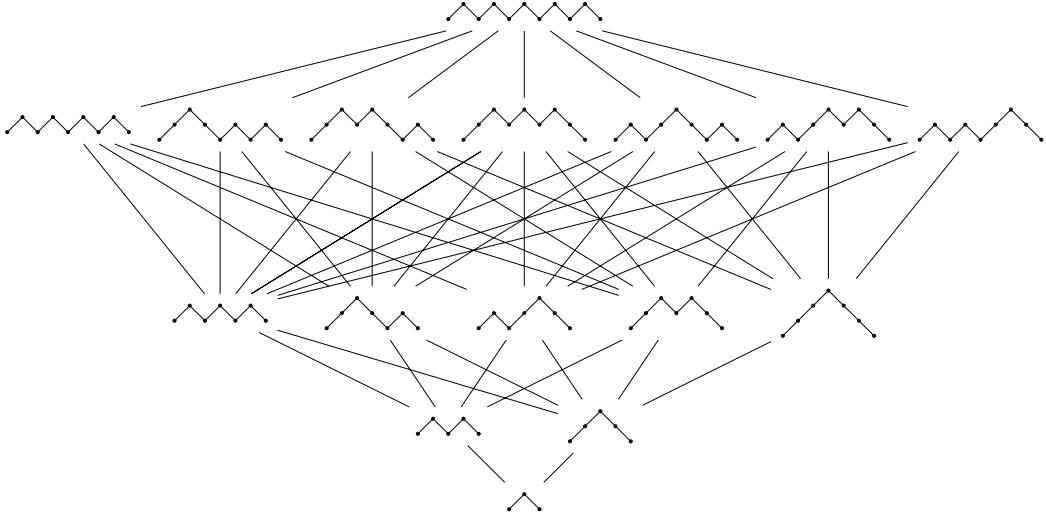

\bigskip
The sequence $(s_0^{(k)}([UD,(UD)^{n}]))_{n\geq k\geq 0}$ appears as A137940 in \cite{S}.
The first few lines of the associated triangle are recorded below (Table \ref{m1}).
\begin{table}
\begin{center}
\begin{tabular}{|c|ccccccccc|}
  \hline
  $n\backslash k$ & 1 & 2 & 3 & 4 & 5 & 6 & 7 & 8 & 9 \\
  \hline
  1 & 1 &  &  &  &  &  &  &  &  \\
  2 & 1 & 1 &  &  &  &  &  &  &  \\
  3 & 1 & 2 & 1 &  &  &  &  &  &  \\
  4 & 1 & 2 & 4 & 1 &  &  &  &  &  \\
  5 & 1 & 2 & 5 & 7 & 1 &  &  &  &  \\
  6 & 1 & 2 & 5 & 13 & 11 & 1 &  &  &  \\
  7 & 1 & 2 & 5 & 14 & 31 & 16 & 1 &  &  \\
  8 & 1 & 2 & 5 & 14 & 41 & 66 & 22 & 1 &  \\
  9 & 1 & 2 & 5 & 14 & 42 & 116 & 127 & 29 & 1 \\
  \hline
\end{tabular}
\end{center}
\caption{Triangle associated with
	sequence A137940}
\label{m1}
\end{table}
As a consequence of the last proposition, and since the Narayana array is symmetric,
the $k$-th diagonal of Table 1 contains the sum of the first $k$ columns of the Narayana array.
Therefore, $s_0^{(k)}([UD,(UD)^{n}])$ counts the number of Dyck paths of semilength $k$ having at most $n-k+1$ peaks.

The sequence $(s_{0}([UD,(UD)^{n}]))_{n\geq 0}$ of the sizes of the intervals $[UD,(UD)^{n}]$ starts
1,2,4,8, 16,33,70,152,337 and is not recorded in \cite{S};
however, it is the sequence of the partial sums of A004148 of \cite{S}, called ``generalized Catalan numbers" and counting, among other things, peak-less Motzkin paths with respect to the length.
In fact, it is not difficult to find a bijective explanation.

\begin{proposition}
There is a bijection between $[UD,(UD)^{n}]$ and the set of peak-less Motzkin paths of length at most $n$.
\end{proposition}

\begin{proof}
Given a peak-less Motzkin path of length $k\leq n$, just replace each of its level steps with a peak:
the resulting Dyck path is easily seen to be in $[UD,(UD)^{n}]$.
Also, it is not difficult to realize that such a map is indeed a bijection.
\end{proof}

\section{The interval $[UD,U^{a+h}D^aU^bD^{b+h}]$}

\subsection{Size of the interval}

Let $a,b,h$ be nonnegative integers. Denote with $Q_{a,b}^{(h)}=U^{a+h}D^aU^bD^{b+h}$ the generic Dyck paths with two peaks.
Our first goal is to find an explicit formula for the cardinality of $\left[UD,Q_{a,b}^{(h)}\right]$, depending on $a,b$ and $h$.
Without loss of generality, we can suppose $b\geq a\geq 1$.

It is clear that each path in the above interval has either one peak or two peaks.
The generic path with only one peak in $\left[UD,Q_{a,b}^{(h)}\right]$ has height $j$, with $1\leq j\leq b+h$.
Therefore there are $b+h$ such paths.
From now on, we will focus on the remaining elements of our interval, i.e. paths having exactly two peaks.

We start with the case $h=0$, that is $Q_{a,b}^{(0)}=U^a D^a U^b D^b$: this is the generic \emph{nonelevated} Dyck path with two peaks.

\begin{proposition}\label{prop3}
Denote with $\varphi_0 (a,b)$ the number of Dyck paths having exactly two peaks in the interval $[UD,Q_{a,b}^{(0)}]$. Then
\begin{equation}\label{phi_0}
\varphi_0 (a,b)=\frac{a(a+1)(3b-a+1)}{6}\ .
\end{equation}
\end{proposition}

\begin{proof}
The generic Dyck path $P$ with two peaks in $[UD,Q_{a,b}^{(0)}]$ can be constructed as follows.
Start by choosing $k$ up steps from the first run of $Q_{a,b}^{(0)}$, for $k$ running from 1 to $a$.
Then choose $t$ down steps from the second run: here it must be $1\leq t\leq k$, since otherwise $P$ would have some points of negative height.
Similarly, the up steps chosen from the third run of $Q_{a,b}^{(0)}$ cannot be too many,
otherwise there would remain too few down steps in the last run to complete the Dyck path.
Specifically, it is possible to select $s$ up steps from the third run of $Q_{a,b}^{(0)}$, with $k-t+s\leq b$
(the quantity $k-t+s$ being the height of the path $P$ at the end of the last ascending run), and so $1\leq s\leq b-k+t$.
Summing up, we get
$$
\varphi_0(a,b)=\sum_{k=1}^{a}\sum_{t=1}^{k}\sum_{s=1}^{b-k+t} 1
$$
which reduces to
$$
\varphi_0 (a,b)=\frac{a(a+1)(3b-a+1)}{6},
$$
as desired.
\end{proof}

The sequence $(\varphi_0 (a,b))_{1\leq a\leq b}$ is A082652 in \cite{S},
where an interpretations in terms of squares inside an $a\times b$ rectangular grid is provided.
We will discuss the connections with our combinatorial setting at the end of the present section.

We are now ready to express the number of Dyck paths having exactly two peaks lying below $Q_{a,b}^{(h)}$.

\begin{proposition}
Denote with $\varphi_h (a,b)$ the number of Dyck paths having exactly two peaks in the interval $[UD,Q_{a,b}^{(h)}]$. Then
\begin{equation}\label{phi_h_def}
\varphi_h (a,b)=\varphi_0 (a,b)+hab\ .
\end{equation}
\end{proposition}

\begin{proof}
Denote with $[UD,Q_{a,b}^{(\ell )}]_2$ the set of Dyck paths in $[UD,Q_{a,b}^{(\ell )}]$ having exactly two peaks.
For each $i=1,\ldots ,h$, define $\mathcal{C}_i =[UD,Q_{a,b}^{(i)}]_2 \setminus [UD,Q_{a,b}^{(i-1)}]_2$.
It is not hard to see that a Dyck path $U^\alpha D^\beta U^\gamma D^\delta \in \mathcal{C}_i$ if and only if
\begin{itemize}
\item $(\alpha ,\beta ,\gamma ,\delta)\leq (i+a,a,b,i+b)$ and
\item $\alpha =i+a$ or $\delta =i+b$,
\end{itemize}
where the above partial order on tuples has to be understood componentwise.

For $i=h$, the above definitions immediately imply that
$$
\varphi_h (a,b)=\varphi_{h-1}(a,b)+|\mathcal{C}_h|.
$$

Iterating this argument, we get:
\begin{equation}\label{fi}
\varphi_h (a,b)=\varphi_0 (a,b)+\sum_{i=1}^{h}|\mathcal{C}_i|.
\end{equation}

Now observe that, if $\Gamma \in \mathcal{C}_i$, then $\Gamma =U^i PD^i$, for some Dyck path $P$ with two peaks.
This is due to the fact that, if the first ascending run of $\Gamma$ has length $i+a$,
the first descending run of $\Gamma$ terminates at height $\geq i$ (because $\Gamma \in [UD,Q_{a,b}^{(h)}]_2$).
Of course an analogous argument holds if the last descending run of $\Gamma$ has length $i+b$.
As a consequence, for each $i=1,\ldots ,h$, we can define a function from $\mathcal{C}_i$ to
$\mathcal{C}_0 =\{ P=U^\alpha D^\beta U^\gamma D^\delta \in [UD,Q_{a,b}^{(0)}]_2 \, |\, \alpha =a \textnormal{ or } \delta =b\}$ which maps $\Gamma =U^i PD^i$ into $P$.
It is easy to see that such a function is well defined (i.e. $P\in \mathcal{C}_0$) and is bijective.
Thus formula (\ref{fi}) can be rewritten as
\begin{equation}\label{phi_h}
\varphi_h (a,b)=\varphi_0 (a,b)+h|\mathcal{C}_0 |.
\end{equation}

Finally, in order to enumerate $\mathcal{C}_0$,
we observe that a path $P\in \mathcal{C}_0$ is uniquely determined by the length $s$ of its first descending run
and the length $t$ of its last ascending run, where $1\leq s\leq a$ and $1\leq t\leq b$.
In fact, the unique path of $\mathcal{C}_0$ corresponding to a legal choice of $s$ and $t$ is
$U^x Q_{s,t}^{(0)}D^x$, where $x=\min \{ a-s,b-t\}$.
Hence we have that $|\mathcal{C}_0 |=ab$. This leads immediately to formula (\ref{phi_h_def}), as desired.
\end{proof}

As we have already noted,
the triangular array determined by $(\varphi_0(a,b))_{1\leq a\leq b}$ is sequence A082652 in \cite{S},
which counts the numbers of squares that can be found in an $a\times b$ rectangular grid.
We now describe a bijection between the set of such squares and the set $[UD,Q_{a,b}^{(0)}]_2$
of Dyck paths having two peaks and which are less than or equal to $Q_{a,b}^{(0)}$.
To this aim,
we encode the Dyck path $U^{k+i}D^iU^{j}D^{j+k}\in [UD,Q_{a,b}^{(0)}]_2$ with the triple $(i,j;k)$.
Observe that $(i,j;k)$ represents a legal Dyck path in $[UD,Q_{a,b}^{(0)}]_2$ if and only if
$i,j\geq 1$, $i+k\leq a$ and $j+k\leq b$.
Moreover,
a triple $(i,j;k)$ satisfying the above conditions also uniquely determines a square
inside an $a\times b$ rectangular grid of unit cells as follows.
Label the rows of the grid with integers $1,2,\ldots ,a$ from top to bottom
and the columns with integers $1,2,\ldots ,b$ from left to right.
A unit cell lying at the intersection between row $i$ and column $i$ will be said in \emph{position} $(i,j)$.
A square inside the grid can be characterized by
providing the position $(i,j)$ of its topmost and leftmost cell and the length $k+1$ of its side.
Notice that
the triple $(i,j;k)$ determines (in a unique way) a square if and only if $i,j\geq 1$, $i+k\leq a$ and $j+k\leq b$.
We have thus found the same encoding both for
Dyck paths in $[UD,Q_{a,b}^{(0)}]_2$ and for squares inside an
 $a\times b$ rectangular grid, which gives the desired bijection.
In Figure \ref{fig2} the path encoded by the triple $(2,3;2)$ is mapped into the corresponding square.
\begin{figure}
	\begin{center}
		\begin{tikzpicture}[scale=0.2]
		\node at (-13,2){$Q_{4,6}^{(0)}\ \ \geq\ \  U^4D^2U^3D^5$\ \ =\ \ };
		\draw [black](0,0)--(2,2);
		\draw[gray!33](2,2)--(4,4)--(6,2);
		\draw[gray!66](6,2)--(9,5)--(12,2);
		\draw [black](12,2)--(14,0);





		\draw[fill] (0,0) circle[radius=0.1];
		\draw[fill] (1,1) circle[radius=0.1];
		\draw[fill] (2,2) circle[radius=0.1];
		\draw[fill] (3,3) circle[radius=0.1];
		\draw[fill] (4,4) circle[radius=0.1];
		\draw[fill] (5,3) circle[radius=0.1];
		\draw[fill] (6,2) circle[radius=0.1];
		\draw[fill] (7,3) circle[radius=0.1];
		\draw[fill] (8,4) circle[radius=0.1];
		\draw[fill] (9,5) circle[radius=0.1];
		\draw[fill] (10,4) circle[radius=0.1];
		\draw[fill] (11,3) circle[radius=0.1];
		\draw[fill] (12,2) circle[radius=0.1];
		\draw[fill] (13,1) circle[radius=0.1];
		\draw[fill] (14,0) circle[radius=0.1];
		\draw[<->] (16,3.5)--(20,3.5);

		\draw[->,line width=1,gray!33] (21,5)--(23,5);
		\draw[->,line width=1,gray!66] (29,10.5)--(29,8.5);


		\draw[step=2cm,gray,very thin] (24,0) grid (36,8);

		\draw[pattern=north west lines,pattern] (28,0) rectangle (34,6);

		\end{tikzpicture}
	\end{center}
	\caption{The path $U^4D^2U^3D^5\leq Q_{4,6}^{(0)}$ and the associated square inside the grid $4\times 6$}
	\label{fig2}
\end{figure}
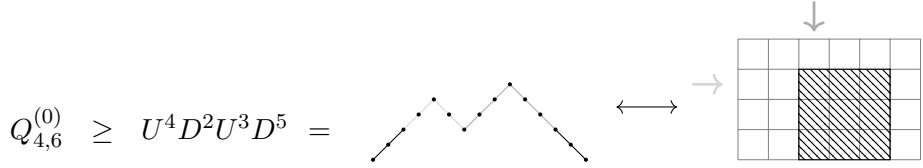

\bigskip

The above bijection can be exploited to transport the order structure on Dyck paths
to a description in terms of squares in a grid.
In the next lemma, we use the above encoding into triples to give an alternative presentation
of the pattern order inside $[UD,Q_{a,b}^{(0)}]_2$.

\begin{lemma}
Set $\Theta =\{ (i,j;k)\, |\, i,j\geq 1, i+k\leq a, j+k\leq b$.
Define a partial order on $\Theta$ by setting $(\alpha ,\beta ;\gamma )\sqsubseteq (i,j;k)$ when
$(\alpha ,\beta ,\gamma )\leq (i,j,\min \{ i+k-\alpha ,j+k-\beta \}$,
where $\leq$ is the usual coordinatewise order on $\mathbb{N}^3$.
Then $(\alpha ,\beta ;\gamma )\sqsubseteq (i,j;k)$ if and only if
$(\alpha ,\beta ;\gamma )\leq (i,j;k)$ in the Dyck pattern poset,
i.e. $U^{\gamma +\alpha}D^{\alpha}U^{\beta}D^{\beta +\gamma}\leq U^{k+i}D^i U^j D^{j+k}$.
\end{lemma}

\begin{proof}
The fact that $U^{\gamma +\alpha}D^{\alpha}U^{\beta}D^{\beta +\gamma}\leq U^{k+i}D^i U^j D^{j+k}$
is equivalent to the system of inequalities
$$
\begin{cases}
\alpha +\gamma \leq i+k\\
\alpha \leq i\\
\beta \leq j\\
\beta +\gamma \leq j+k
\end{cases}
$$
since the two paths have the same number of peaks,
and so the steps of each run of the smaller path must be selected
from the steps of the corresponding run of the larger one.
The above inequalities can be equivalently written as
$(\alpha ,\beta ,\gamma )\leq (i,j,\min \{ i+k-\alpha ,j+k-\beta \}$,
which is exactly $(\alpha ,\beta ;\gamma )\sqsubseteq (i,j;k)$.
\end{proof}

Using the above lemma and bijection,
the pattern order on $[UD,Q_{a,b}^{(0)}]_2$ can now be expressed as a partial order on the set of
squares inside an $a\times b$ rectangular grid.
Denote with $\underline{xy}|$ the rectangle having a pair of opposite corners in positions $(1,1)$ and $(x,y)$.
Take two squares $Q,Q'$ in the grid
whose topmost and leftmost cells are in positions $(\alpha ,\beta )$ and $(i,j)$, respectively,
and whose sides have length $\gamma +1$ and $k+1$;
hence they are encoded by the triples $Q=(\alpha ,\beta ;\gamma )$ and $Q'=(i,j;k)$.
The partial order on $\Theta$ defined in the above lemma can be read off on squares in the following way:
$Q\leq Q'$ when the topmost and leftmost cell of $Q$ is in the rectangle $\underline{ij}|$
and the opposite cell is in the rectangle $\underline{(i+k)(j+k)}|$.

%
%
%
%
%
%

\bigskip

Suppose now $r\in \mathbb{N}$ and $2\leq r\leq a+b$.
We will refine our previous enumerative result by counting the number of elements in $[UD,Q_{a,b}^{(h)}]$ with semilength $r$.
This clearly gives the rank distribution of the elements of $[UD,Q_{a,b}^{(h)}]$. As an example, Figure \ref{fig3} shows the interval $[UD,Q_{2,3}^{(1)}]$.

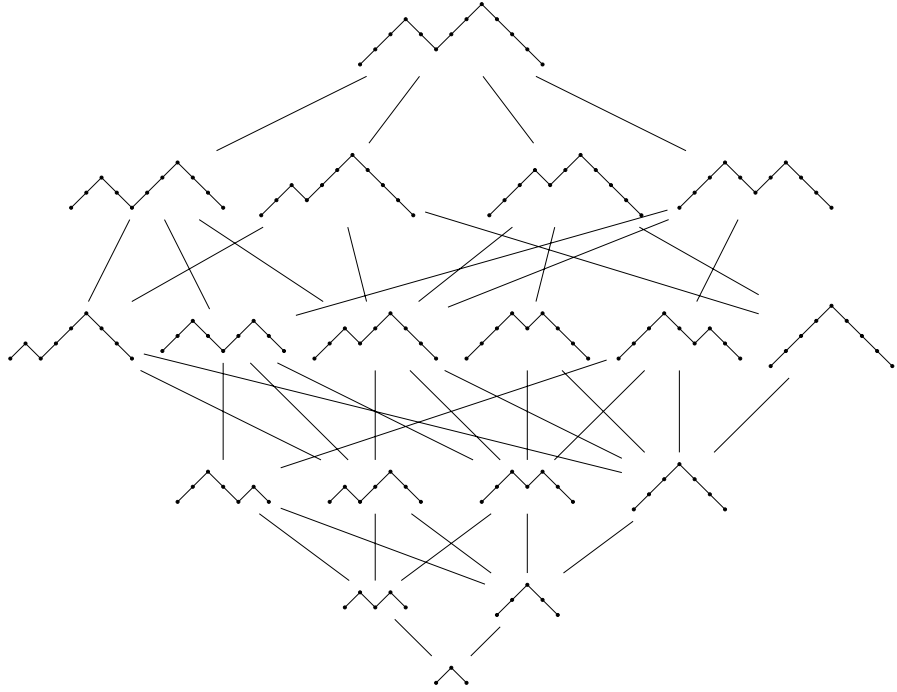
\begin{figure}[h]
	\begin{center}
		{\begin{tikzpicture}
			\node (1) at (0,0)
			{\scalebox{0.2}
				{\begin{tikzpicture}
					\draw  (0,0) --  (1,1) -- (2,0);
					\draw[fill] (0,0) circle[radius=0.1];
					\draw[fill] (1,1) circle[radius=0.1];
					\draw[fill] (2,0) circle[radius=0.1];
					\end{tikzpicture}}};
			\node (2) at (-1,1)
			{\scalebox{0.2}
				{\begin{tikzpicture}
					\draw  (0,0) --  (1,1) -- (2,0) -- (3,1) -- (4,0);
					\draw[fill] (0,0) circle[radius=0.1];
					\draw[fill] (1,1) circle[radius=0.1];
					\draw[fill] (2,0) circle[radius=0.1];
					\draw[fill] (3,1) circle[radius=0.1];
					\draw[fill] (4,0) circle[radius=0.1];			
					\end{tikzpicture}}};
			\node (3) at (1,1)
			{\scalebox{0.2}
				{\begin{tikzpicture}
					\draw  (0,0) --  (2,2) --  (4,0);
					\draw[fill] (0,0) circle[radius=0.1];
					\draw[fill] (1,1) circle[radius=0.1];
					\draw[fill] (2,2) circle[radius=0.1];
					\draw[fill] (3,1) circle[radius=0.1];
					\draw[fill] (4,0) circle[radius=0.1];
					\end{tikzpicture}}};
			\node (4) at (-3,2.5)
			{\scalebox{0.2}
				{\begin{tikzpicture}
					\draw  (0,0) --  (2,2) --  (4,0) -- (5,1) --
					(6,0);
					\draw[fill] (0,0) circle[radius=0.1];
					\draw[fill] (1,1) circle[radius=0.1];
					\draw[fill] (2,2) circle[radius=0.1];
					\draw[fill] (3,1) circle[radius=0.1];
					\draw[fill] (4,0) circle[radius=0.1];
					\draw[fill] (5,1) circle[radius=0.1];
					\draw[fill] (6,0) circle[radius=0.1];
					\end{tikzpicture}}};
			\node (5) at (-1,2.5)
			{\scalebox{0.2}
				{\begin{tikzpicture}
					\draw  (0,0) --  (1,1) --  (2,0) -- (4,2) --
					(6,0);
					\draw[fill] (0,0) circle[radius=0.1];
					\draw[fill] (1,1) circle[radius=0.1];
					\draw[fill] (2,0) circle[radius=0.1];
					\draw[fill] (3,1) circle[radius=0.1];
					\draw[fill] (4,2) circle[radius=0.1];
					\draw[fill] (5,1) circle[radius=0.1];
					\draw[fill] (6,0) circle[radius=0.1];
					\end{tikzpicture}}};
			\node (6) at (1,2.5)
			{\scalebox{0.2}
				{\begin{tikzpicture}
					\draw  (0,0) --  (2,2) --  (3,1) -- (4,2) --
					(6,0);
					\draw[fill] (0,0) circle[radius=0.1];
					\draw[fill] (1,1) circle[radius=0.1];
					\draw[fill] (2,2) circle[radius=0.1];
					\draw[fill] (3,1) circle[radius=0.1];
					\draw[fill] (4,2) circle[radius=0.1];
					\draw[fill] (5,1) circle[radius=0.1];
					\draw[fill] (6,0) circle[radius=0.1];
					\end{tikzpicture}}};
			\node (7) at (3,2.5)
			{\scalebox{0.2}
				{\begin{tikzpicture}
					\draw  (0,0) --  (3,3) -- (6,0);
					\draw[fill] (0,0) circle[radius=0.1];
					\draw[fill] (1,1) circle[radius=0.1];
					\draw[fill] (2,2) circle[radius=0.1];
					\draw[fill] (3,3) circle[radius=0.1];
					\draw[fill] (4,2) circle[radius=0.1];
					\draw[fill] (5,1) circle[radius=0.1];
					\draw[fill] (6,0) circle[radius=0.1];			
					\end{tikzpicture}}};
			\node (8) at (-5,4.5)
			{\scalebox{0.2}
				{\begin{tikzpicture}
					\draw  (0,0) -- (1,1) -- (2,0) -- (5,3) -- (8,0) ;
					\draw[fill] (0,0) circle[radius=0.1];
					\draw[fill] (1,1) circle[radius=0.1];
					\draw[fill] (2,0) circle[radius=0.1];
					\draw[fill] (3,1) circle[radius=0.1];
					\draw[fill] (4,2) circle[radius=0.1];
					\draw[fill] (5,3) circle[radius=0.1];
					\draw[fill] (6,2) circle[radius=0.1];
					\draw[fill] (7,1) circle[radius=0.1];
					\draw[fill] (8,0) circle[radius=0.1];
					\end{tikzpicture}}};
			\node (9) at (-3,4.5)
			{\scalebox{0.2}
				{\begin{tikzpicture}
					\draw  (0,0) -- (2,2) -- (4,0) -- (6,2) -- (8,0);
					\draw[fill] (0,0) circle[radius=0.1];
					\draw[fill] (1,1) circle[radius=0.1];
					\draw[fill] (2,2) circle[radius=0.1];
					\draw[fill] (3,1) circle[radius=0.1];
					\draw[fill] (4,0) circle[radius=0.1];
					\draw[fill] (5,1) circle[radius=0.1];
					\draw[fill] (6,2) circle[radius=0.1];
					\draw[fill] (7,1) circle[radius=0.1];
					\draw[fill] (8,0) circle[radius=0.1];
					\end{tikzpicture}}};
			\node (10) at (-1,4.5)
			{\scalebox{0.2}
				{\begin{tikzpicture}
					\draw  (0,0) -- (2,2) -- (3,1) -- (5,3) -- (8,0);
					\draw[fill] (0,0) circle[radius=0.1];
					\draw[fill] (1,1) circle[radius=0.1];
					\draw[fill] (2,2) circle[radius=0.1];
					\draw[fill] (3,1) circle[radius=0.1];
					\draw[fill] (4,2) circle[radius=0.1];
					\draw[fill] (5,3) circle[radius=0.1];
					\draw[fill] (6,2) circle[radius=0.1];
					\draw[fill] (7,1) circle[radius=0.1];
					\draw[fill] (8,0) circle[radius=0.1];
					\end{tikzpicture}}};
			\node (11) at (1,4.5)
			{\scalebox{0.2}
				{\begin{tikzpicture}
					\draw  (0,0) -- (3,3) -- (4,2) -- (5,3) -- (8,0);
					\draw[fill] (0,0) circle[radius=0.1];
					\draw[fill] (1,1) circle[radius=0.1];
					\draw[fill] (2,2) circle[radius=0.1];
					\draw[fill] (3,3) circle[radius=0.1];
					\draw[fill] (4,2) circle[radius=0.1];
					\draw[fill] (5,3) circle[radius=0.1];
					\draw[fill] (6,2) circle[radius=0.1];
					\draw[fill] (7,1) circle[radius=0.1];
					\draw[fill] (8,0) circle[radius=0.1];			
					\end{tikzpicture}}};
			\node (12) at (3,4.5)
			{\scalebox{0.2}
				{\begin{tikzpicture}
					\draw  (0,0) -- (3,3) -- (5,1) -- (6,2) -- (8,0);
					\draw[fill] (0,0) circle[radius=0.1];
					\draw[fill] (1,1) circle[radius=0.1];
					\draw[fill] (2,2) circle[radius=0.1];
					\draw[fill] (3,3) circle[radius=0.1];
					\draw[fill] (4,2) circle[radius=0.1];
					\draw[fill] (5,1) circle[radius=0.1];
					\draw[fill] (6,2) circle[radius=0.1];
					\draw[fill] (7,1) circle[radius=0.1];
					\draw[fill] (8,0) circle[radius=0.1];			
					\end{tikzpicture}}};
			\node (13) at (5,4.5)
			{\scalebox{0.2}
				{\begin{tikzpicture}
					\draw  (0,0) -- (4,4) -- (8,0);
					\draw[fill] (0,0) circle[radius=0.1];
					\draw[fill] (1,1) circle[radius=0.1];
					\draw[fill] (2,2) circle[radius=0.1];
					\draw[fill] (3,3) circle[radius=0.1];
					\draw[fill] (4,4) circle[radius=0.1];
					\draw[fill] (5,3) circle[radius=0.1];
					\draw[fill] (6,2) circle[radius=0.1];
					\draw[fill] (7,1) circle[radius=0.1];
					\draw[fill] (8,0) circle[radius=0.1];	
					\end{tikzpicture}}};
			\node (14) at (-4,6.5)
			{\scalebox{0.2}
				{\begin{tikzpicture}
					\draw  (0,0) -- (2,2) -- (4,0) -- (7,3) -- (10,0);
					\draw[fill] (0,0) circle[radius=0.1];
					\draw[fill] (1,1) circle[radius=0.1];
					\draw[fill] (2,2) circle[radius=0.1];
					\draw[fill] (3,1) circle[radius=0.1];
					\draw[fill] (4,0) circle[radius=0.1];
					\draw[fill] (5,1) circle[radius=0.1];
					\draw[fill] (6,2) circle[radius=0.1];
					\draw[fill] (7,3) circle[radius=0.1];
					\draw[fill] (8,2) circle[radius=0.1];
					\draw[fill] (9,1) circle[radius=0.1];
					\draw[fill] (10,0) circle[radius=0.1];
					\end{tikzpicture}}};
			\node (15) at (-1.5,6.5)
			{\scalebox{0.2}
				{\begin{tikzpicture}
					\draw  (0,0) -- (2,2) -- (3,1) -- (6,4) -- (10,0);
					\draw[fill] (0,0) circle[radius=0.1];
					\draw[fill] (1,1) circle[radius=0.1];
					\draw[fill] (2,2) circle[radius=0.1];
					\draw[fill] (3,1) circle[radius=0.1];
					\draw[fill] (4,2) circle[radius=0.1];
					\draw[fill] (5,3) circle[radius=0.1];
					\draw[fill] (6,4) circle[radius=0.1];
					\draw[fill] (7,3) circle[radius=0.1];
					\draw[fill] (8,2) circle[radius=0.1];
					\draw[fill] (9,1) circle[radius=0.1];
					\draw[fill] (10,0) circle[radius=0.1];
					\end{tikzpicture}}};
			\node (16) at (1.5,6.5)
			{\scalebox{0.2}
				{\begin{tikzpicture}
					\draw  (0,0) -- (3,3) -- (4,2) -- (6,4) -- (10,0);
					\draw[fill] (0,0) circle[radius=0.1];
					\draw[fill] (1,1) circle[radius=0.1];
					\draw[fill] (2,2) circle[radius=0.1];
					\draw[fill] (3,3) circle[radius=0.1];
					\draw[fill] (4,2) circle[radius=0.1];
					\draw[fill] (5,3) circle[radius=0.1];
					\draw[fill] (6,4) circle[radius=0.1];
					\draw[fill] (7,3) circle[radius=0.1];
					\draw[fill] (8,2) circle[radius=0.1];
					\draw[fill] (9,1) circle[radius=0.1];
					\draw[fill] (10,0) circle[radius=0.1];
					\end{tikzpicture}}};
			\node (17) at (4,6.5)
			{\scalebox{0.2}
				{\begin{tikzpicture}
					\draw  (0,0) -- (3,3) -- (5,1) -- (7,3) -- (10,0);
					\draw[fill] (0,0) circle[radius=0.1];
					\draw[fill] (1,1) circle[radius=0.1];
					\draw[fill] (2,2) circle[radius=0.1];
					\draw[fill] (3,3) circle[radius=0.1];
					\draw[fill] (4,2) circle[radius=0.1];
					\draw[fill] (5,1) circle[radius=0.1];
					\draw[fill] (6,2) circle[radius=0.1];
					\draw[fill] (7,3) circle[radius=0.1];
					\draw[fill] (8,2) circle[radius=0.1];
					\draw[fill] (9,1) circle[radius=0.1];
					\draw[fill] (10,0) circle[radius=0.1];
					\end{tikzpicture}}};
			\node (18) at (0,8.5)
			{\scalebox{0.2}
				{\begin{tikzpicture}
					\draw  (0,0) -- (3,3) -- (5,1) -- (8,4) -- (12,0);
					\draw[fill] (0,0) circle[radius=0.1];
					\draw[fill] (1,1) circle[radius=0.1];
					\draw[fill] (2,2) circle[radius=0.1];
					\draw[fill] (3,3) circle[radius=0.1];
					\draw[fill] (4,2) circle[radius=0.1];
					\draw[fill] (5,1) circle[radius=0.1];
					\draw[fill] (6,2) circle[radius=0.1];
					\draw[fill] (7,3) circle[radius=0.1];
					\draw[fill] (8,4) circle[radius=0.1];
					\draw[fill] (9,3) circle[radius=0.1];
					\draw[fill] (10,2) circle[radius=0.1];
					\draw[fill] (11,1) circle[radius=0.1];
					\draw[fill] (12,0) circle[radius=0.1];
					\end{tikzpicture}}};
			\draw[very thin] (1) -- (2) (1) -- (3) (2)--(4) (2)--(5) (2)--(6) (3)--(4) (3)--(5) (3)--(6) (3)--(7)  (8)--(5) (8)--(7) (9)--(4) (9)--(5) (9)--(6) (10)--(5) (10)--(6) (10)--(7) (11)--(6) (11)--(7) (12)--(4) (12)--(6) (12)--(7) (13)--(7) (14)--(8) (14)--(9) (14)--(10) (15)--(8) (15)--(10) (15)--(13) (16)--(10) (16)--(11) (16)--(13) (17)--(9) (17)--(10) (17)--(12) (18)--(14) (18)--(15) (18)--(16) (18)--(17);
	\end{tikzpicture}}
\end{center}
\caption{The Hasse diagram of the interval $[UD,Q_{2,3}^{(1)}]$ in
the Dyck pattern poset.}
\label{fig3}
\end{figure}

\begin{proposition}
The number of elements of $[UD,Q_{a,b}^{(h)}]$ having semilength $r$ is given by
\begin{equation}\label{rank_2_peaks}
s_0 ^{(r)}[UD,Q_{a,b}^{(h)}] =
\sum_{i=\max \{ 1,r-b-h\} }^{\min \{ a,r-1\} }\left( \min \{ b,r-i\} -\max \{ 1,r-a-h\}+1 \right) + [r\leq b+h]\ \ ,
\end{equation}
where $[\Omega ]$ denotes the characteristic function of the property $\Omega$,
that is $[\Omega ] =1$ if $\Omega$ is true and $[\Omega ] =0$ if $\Omega$ is false.
\end{proposition}

\begin{proof}
We can clearly limit ourselves to considering paths having two peaks,
since paths with just one peak inside $[UD,Q_{a,b}^{(h)}]$ are easily counted
(we have exactly one such path for any rank $r\leq b+h$).

Using an approach similar to that of Proposition \ref{prop3},
we observe that a generic Dyck path $P=U^{k+i}D^iU^jD^{j+k}\in [UD,Q_{a,b}^{(h)}]_2$ of semilength $r$
can be constructed as follows.
Start by choosing $k+i$ up steps from the first run of $Q_{a,b}^{(h)}$, with $k+i\leq a+h$.
Then choose $i$ down steps from $Q_{a,b}^{(h)}$ for the second run of $P$, so that $1\leq i \leq a$,
since in the second run of $Q_{a,b}^{(h)}$ there are exactly $a$ down steps.
For the third run of $P$, it must be $1\leq j \leq b$ (since the third run of $Q_{a,b}^{(h)}$ has length $b$)
and $j+k\leq b+h$
(since the height of $P$ at the end of the third run cannot exceed $b+h$,
which is the maximum possible length of the fourth run of $P$).
Finally, we have $i+j+k=r$, so that the following relations hold:

$$
\begin{cases}
i+k \leq a+h\\
1\leq i \leq a\\
1\leq j\leq b\\
j+k\leq b+h\\
i+j+k=r\ .
\end{cases}
$$

From the last relation, we have $i+k=r-j$ and $j+k=r-i$,
so that $i\leq r-j$ and $j\leq r-i$.
Therefore, after some manipulations, the above inequalities can be equivalently written as follows,
in order to make as explicit as possible the ranges of $i$ and $j$
(note that once $i$ and $j$ are fixed, the term $k$ is uniquely determined):

$$
\begin{cases}
j \geq r-a-h\\
1\leq i \leq a\\
1\leq j\leq b\\
i\geq r-b-h\\
i\leq r-j\leq r-1\\
j\leq r-i\\
k=r-i-j\ .
\end{cases}
$$

These inequalities can be condensed in

$$
\begin{cases}
	\max\{1,r-b-h\} \leq i \leq \min\{a,r-1\}\\
	\max\{1,r-a-h\} \leq j \leq \min\{b,r-i\}\\
	k=r-i-j.
\end{cases}
$$

The above computations immediately lead to a formula for the rank distribution of the elements of $[UD,Q_{a,b}^{(h)}]$:

\begin{align*}
s_0 ^{(r)}[UD,Q_{a,b}^{(h)}] &=
[r\leq b+h] + \sum_{i=\max \{ 1,r-b-h\} }^{\min \{a,r-1\} }\ \
\sum_{j=\max\{1,r-a-h\}}^{\max\{b,r-i\}} 1\\
&\\
&=\sum_{i=\max \{ 1,r-b-h)\} }^{\min \{ a,r-1\} }\left( \min \{ b,r-i\} -\max \{ 1,r-a-h\}+1 \right) + [r\leq b+h]\ \ .
\end{align*}

\end{proof}

Formula \ref{rank_2_peaks} is not very easy to read as it is written.
However, if we assign specific values to $a,b$ or $h$, some much nicer expressions can be obtained.
For instance, in the particular case $h=0$, depending on the value of the semilength $r$, we get what follows.

\begin{itemize}
	\item If $2\leq r\leq a+1$, then $\min\{a,r-1\}=r-1$, $\max\{1,r-b\}=1$, $\min\{r-i,b\}=r-i$, and $\max\{1,r-a\}=1$. Therefore
	$$
	s_0 ^{(r)}[UD,Q_{a,b}^{(0)}]_2=
	\sum_{i=1}^{r-1}(r-i)=\sum_{i=1}^{r-1}i={r\choose 2}.
	$$
	\item If $a+1\leq r\leq b$, then $\min\{a,r-1\}=a$, $\max\{1,r-b\}=1$, $\min\{r-i,b\}=r-i$, and $\max\{1,r-a\}=r-a$. Therefore
	$$
	s_0 ^{(r)}[UD,Q_{a,b}^{(0)}]_2=
	\sum_{i=1}^a(r-i-r+a+1)=\sum_{i=1}^a(a+1-i)=
	\sum_{i=1}^a i={a+1\choose 2}.
	$$
	\item If $b+1\leq r \leq a+b$, then $\min\{a,r-1\}=a$, $\max\{1,r-b\}=r-b$, $\min\{r-i,b\}=r-i$, and $\max\{1,r-a\}=r-a$. Therefore
	$$
	s_0 ^{(r)}[UD,Q_{a,b}^{(0)}]_2=
	\sum_{i=r-b}^a(r-i-r+a+1)=\sum_{i=r-b}^a(a+1-i)=
	\sum_{i=1}^{a+b-r+1} i={a+b-r+2\choose 2}.
	$$
\end{itemize}

%
%

Setting $m=\min\{r-1,a,a+b-r+1\}$, it is not difficult to show that $2\leq r\leq a+1$ ($a+1\leq r\leq b$, $b+1\leq r\leq a+b$, respectively) if and only if $m=r-1$ ($m=a$, $m=a+b-r+1$, respectively).
As a consequence
$$
s_0 ^{(r)}[UD,Q_{a,b}^{(0)}]_2=
\sum_{i=1}^m i={m+1 \choose 2},
$$
hence
$$
s_0 ^{(r)}[UD,Q_{a,b}^{(0)}]={m+1 \choose 2}+[r\leq b].
$$

\subsection{Enumeration of covering relations}

In this section we work out a formula for the number $s_{1}([UD,Q_{a,b}^{(0)}])$ of edges in the Hasse diagram of the interval $[UD,Q_{a,b}^{(0)}]$;
as usual, we can assume w.l.o.g. that $1\leq a\leq b$.
For this purpose, we will use the following lemma.

\begin{lemma}\label{delta} Let $(i,j,k)\in \mathbb{N}^{3}$.
Recall that we denote with $\Delta (Q_{i,j}^{(h)})$ the number of paths covered by $Q_{i,j}^{(h)}$.
Then
\begin{equation}\label{rank}
\Delta (Q_{i,j}^{(h)})=\begin{cases}
1 & (i,j,k)\in \{(1,1,0)
\}\\
2 & (i,j,k)\in \{(i,1,0),(1,j,0),(1,1,k)\, |\, i,j\geq 2,\ k\geq 1\}\\
3 & (i,j,k)\in \{(i,j,0),(i,1,k),(1,j,k)\, |\, i,j\geq 2,\ k\geq 1\}\\
4 & (i,j,k)\in\{(i,j,k)\, |\,  i,j\geq 2,\ k\geq 1\}\end{cases}
\end{equation}
\end{lemma}

\begin{proof}
A Dyck path having two peaks can cover at most 4 paths,
obtained by removing one up step from one of the two ascending runs and one down step from one of the two descending runs.
However, in some cases different choices can lead to the same path.
Using Proposition 2.1 of \cite{BBFGPW}, we get exactly the formula given in the statement.
\end{proof}


The above lemma, together with the following formula
\begin{equation}\label{cover_count}
s_{1}([UD,Q_{a,b}^{(0)}])=\sum_{n\geq 0}n\cdot \Delta_{n}([UD,Q_{a,b}^{(0)}]),
\end{equation}
where $\Delta_n ([UD,Q_{a,b}^{(0)}])=\{ P\in [UD,Q_{a,b}^{(0)}]\, |\, \Delta (P)=n\}$, allows to find the desired enumeration.

Clearly we have $\Delta(P)=1$ if and only if either $P=(UD)^{2}$ or $P= U^{i}D^{i}$, for some $i\in \{2,...,b\}$,
hence $\Delta_{1}([UD,Q_{a,b}^{(0)}])=1+(b-1)=b$.

Now take $(i,j,k)\in \mathbb{N}^{3}$, with $i,j\geq 2$ and $k\geq 1$.
Recalling relations (\ref{rank}), we are able to compute $\Delta_{n}([UD,Q_{a,b}^{(0)}])$ for $n=2,3,4$.

When $n=2$, we have that
$$
\begin{cases}
Q_{i,1}^{(0)}\leq Q_{a,b}^{(0)} \textnormal{ if and only if } 2\leq i\leq a;\\
Q_{1,j}^{(0)}\leq Q_{a,b}^{(0)} \textnormal{ if and only if } 2\leq j\leq b;\\
Q_{1,1}^{(k)}\leq Q_{a,b}^{(0)} \textnormal{ if and only if } 1\leq k\leq a-1.\\
\end{cases}
$$

Therefore it follows from Lemma \ref{delta} that $\Delta_{2}([UD,Q_{a,b}^{(0)}])=(a-1)+(b-1)+(a-1)=2a+b-3$.

When $n=3$, we have that
$$
\begin{cases}
Q_{i,j}^{(0)}\leq Q_{a,b}^{(0)} \textnormal{ if and only if } 2\leq i\leq a,\, 2\leq j\leq b;\\
Q_{i,1}^{(k)}\leq Q_{a,b}^{(0)} \textnormal{ if and only if } 2\leq i\leq a-k,\, 1\leq k\leq a-1;\\
Q_{1,j}^{(k)}\leq Q_{a,b}^{(0)} \textnormal{ if and only if } 2\leq j\leq b-k,\, 1\leq k\leq a-1.\\
\end{cases}
$$

Again using Lemma \ref{delta} and some standard computations, we get
\begin{align*}
\Delta_{3}([UD,Q_{a,b}^{(0)}]) & =(a-1)(b-1)+\sum_{k=1}^{a-1}\left( (a-k-1)+(b-k-1)\right) \\
                        & =(a-1)(b-1)+(a+b-2)(a-1)-a(a-1)\\
                        & =(a-1)(2b-3).
\end{align*}

Finally, when $n=4$, we have that $Q_{i,j}^{(k)}\leq Q_{a,b}^{(0)}$ if and only if $1\leq k\leq a-1$, $2\leq i\leq a-k$ and $2\leq j\leq b-k$;
thus we obtain
\begin{align*}
\Delta_{4}([UD,Q_{a,b}^{(0)}]) & =\sum_{k=1}^{a-1}(a-k-1)(b-k-1)\\
                        & =\sum_{k=1}^{a-1}\left( (a-1)(b-1)-(a+b-2)k+k^{2}\right) \\
                        & =(a-1)^{2}(b-1)-(a+b-2)\binom{a}{2}+\frac{a(a-1)(2a-1)}{6}.\\
\end{align*}

Using formula (\ref{cover_count}), we thus have the following.

\begin{proposition}
The number of covering relation in the interval $[UD,Q_{a,b}^{(0)}]$ is given by
\begin{equation}
s_{1}([UD,P(a,b)])=-\frac{1}{3}(2a^{3}-6a^{2}b+a-3b+3).
\end{equation}
\end{proposition}

\subsection{The M\"obius function}

To conclude our analysis of the intervals $[UD,Q_{a,b}^{(h)}]$ we now completely determine their M\"obius function.
Note that the M\"obius function of intervals of the form $[UD,U^n D^n ]$ is trivial 
since these intervals consist of a single chain:
\begin{itemize}
\item if $n=1$, the M\"obius function is $1$;
\item if $n=2$, the M\"obius function is $-1$;
\item otherwise, the M\"obius function is $0$.
\end{itemize}

The next proposition shows that $\mu (UD,Q_{a,b}^{(h)})$ is almost always 0.
As usual, we assume that $a\leq b$.

\begin{proposition}
If at least one among $h$ and $b-a$ is strictly bigger than 1, then $\mu (UD,Q_{a,b}^{(h)})=0$.
\end{proposition}

\begin{proof}
We use induction on the semilength $r=a+b+h$ of $Q_{a,b}^{(h)}$.

It is easy to see that, if the hypothesis of the proposition is satisfied, then $r\geq 4$.
Moreover, if $r=4$, then either $a=b=1$ and $h=2$ or $a=1$, $b=3$ and $h=0$.
In such cases, the maximum of $\mu (UD,Q_{a,b}^{(h)})$ is either $U^3 DUD^3$ or $UDU^3 D^3$, respectively.
and it is immediate to verify that $\mu (UD,U^3 DUD^3)=\mu (UD,UDU^3 D^3)=0$.

Now suppose that $r>4$.
Consider the longest Dyck path $Q_{i,j}^{(k)}<Q_{a,b}^{(h)}$ such that $k,j-i\leq 1$.
Such a path can be explicitly described as follows:
\begin{itemize}
\item if $h\leq 1$, then necessarily $a+1<b$, so we set $k=h$, $i=a$ and $j=a+1$;
\item if $h>1$ and $a=b$, then we set $k=1$ and $i=j=a$;
\item if $h>1$ and $a<b$, then we set $k=1$, $i=a$ and $j=a+1$.
\end{itemize}

From the above construction, it is clear that $Q_{i,j}^{(k)}$ is the longest path in the interval $[UD,Q_{a,b}^{(h)}]$
for which the hypothesis of the proposition does not hold.
In other words, if $Z<Q_{a,b}^{(h)}$ is a Dyck path which does not satisfies the hypothesis of the proposition,
i.e. such that the absolute value of the difference of the lengths of every pair of consecutive runs is $\leq 1$,
then $Z\leq Q_{i,j}^{(k)}$.

Now take a path $Z\nleq Q_{i,j}^{(k)}$, $Z\neq Q_{a,b}^{(h)}$.
By the inductive hypothesis, we then have that $\mu (UD,Z)=0$.
So we can now compute the M\"obius function of $[UD,Q_{a,b}^{(h)}]$ as follows:
$$
\mu (UD,Q_{a,b}^{(h)})=-\sum_{Z\leq Q_{i,j}^{(k)}}\mu (UD,Z)-
\sum_{Z<Q_{a,b}^{(h)}\atop Z\nleq Q_{i,j}^{(k)}}\mu (UD,Z).
$$

In the r.h.s, the first sum is 0 by definition of M\"obius function,
and the second sum is 0 as well, since each of its summand is 0.
This concludes the proof.
\end{proof}

Now, to conclude the computation of the M\"obius function when the maximum of the interval has exactly two peaks,
we have to analyze a few remaining cases.

\begin{proposition}
If $h,b-a\leq 1$, then we are in one of the following cases:
\begin{itemize}
\item $\mu (UD,Q_{a,a+1}^{(1)})=-1$ (with $a\geq 1$);
\item $\mu (UD,Q_{a,a+1}^{(0)})=1$ (with $a\geq 1$);
\item $\mu (UD,Q_{a,a}^{(0)})=-2$, if $a\geq 2$, moreover, $\mu (UD,Q_{1,1}^{(0)})=-1$;
\item $\mu (UD,Q_{a,a}^{(1)})=2$, if $a\geq 2$, moreover, $\mu (UD,Q_{1,1}^{(1)})=1$.
\end{itemize}
\end{proposition}

\begin{proof} All the small cases (namely, when $a\leq 2$) can be easily checked with a simple computation.
When $a>2$, we can proceed by induction on the semilength of the top of the interval.
So, suppose for instance that the top of the interval is $Q_{a,a+1}^{(1)}$ (the first of the above listed cases).
Then, among the paths covered by $Q_{a,a+1}^{(1)}$, there is $Q_{a,a}^{(1)}$, and we can write:
$$
\mu (UD,Q_{a,a+1}^{(1)})=
-\sum_{Z\leq Q_{a,a}^{(1)}}\mu (UD,Z)-\sum_{Z\leq Q_{a,a+1}^{(1)}\atop Z\nleq Q_{a,a}^{(1)}}\mu (UD,Z).
$$

The first sum of the r.h.s is of course 0.
To evaluate the second sum,
we need to find all paths $Z\leq Q_{a,a+1}^{(1)},Z\nleq Q_{a,a}^{(1)}$ such that
the absolute value of the difference between the lengths of any two consecutive runs is at most one
(otherwise, thanks to the previous proposition, the contribution to the above sum is 0).
It is not difficult to realize that, in the case under consideration,
the only path with the required properties is $Q_{a-1,a}^{(0)}$.
By induction, we know that $\mu (UD,Q_{a-1,a}^{(0)})=1$,
and so we can conclude that $\mu (UD,Q_{a,a+1}^{(1)})=-1$, as desired.
The three remaining cases can be dealt with using analogous arguments.
\end{proof}

\section{Again on the M\"obius function and further work}

The combinatorics of the intervals of the Dyck pattern poset is still largely unknown.
We have just provided the first results in this directions,
concerning the enumerative combinatorics of specific intervals (cardinality and covering relations),
as well as the computation of the M\"obius function in a special case.
Concerning this last topic, we can prove some further results,
which give some insight on this important invariant.

First of all, the absolute value of the M\"obius function of the Dyck pattern poset is unbounded.
This is a consequence of the following.

\begin{proposition}
For all $n\geq 2$, $\mu ((UD)^{n-1},(UD)^{n+1})={n\choose 2}$.
\end{proposition}

\begin{proof}
Since the interval $I_n=[(UD)^{n-1},(UD)^{n+1}]$ has rank 2,
its M\"obius function is simply given by
the number of its elements covering the minimum (or covered by the maximum) minus 1.
Any path greater than $(UD)^{n-1}$ has at least $n-1$ peaks.
Any path smaller than $(UD)^{n+1}$ has at most $n$ peaks;
moreover, if it has exactly $n$ peaks, it is necessarily the path $(UD)^n$.
In order to count the paths inside $I_n$ having exactly $n-1$ peaks,
we observe that they can be obtained from $(UD)^{n-1}$
by just adding a new up step to one of the peaks and a new down step to one of the peaks as well
(possibly the same one), in such a way that the resulting path is still Dyck.
It is not difficult to realize that this is equivalent to
choosing a multiset having 2 elements out of a set having $n-1$ elements,
which can be done in ${n\choose 2}$ ways.
\end{proof}

We can also determine the maximum value of the M\"obius function on intervals of rank 2.

\begin{proposition}
For all $n\geq 1$, $\mu (U(UD)^{n-1}D,U(UD)^{n+1}D)=n^2$,
and this is the maximum value attained by $\mu$ on intervals of rank 2.
\end{proposition}

\begin{proof}
We start by observing that, for a given path $Q$ of semilength $n$,
the maximum number of paths covering it is $n^2 +1$.
In fact, if $Q$ has $k$ factors having semilengths $f_1 ,f_2 ,\ldots f_k$, respectively,
then, using Proposition 2.2 of \cite{BFPW}, we get that the number of paths covering $Q$ is
$$
1+\sum_{i}f_i ^2 +\sum_{i<j}f_i f_j =1+\left( \sum_{i}f_i \right) ^2 -\sum_{i<j}f_i f_j =1+n^2 -\sum_{i<j}f_i f_j .
$$

The maximum of the above quantity is indeed $n^2 +1$ and is attained when $\sum_{i<j}f_i f_j$ (which corresponds to having only one factor).
To finish the proof, it will be enough to show that, for the interval in the statement of the proposition,
all the paths covering the lower path are also covered by the upper path.
This can be done quite easily, by means of a case-by-case analysis
(just insert in the lower path a $U$ and a $D$ in all possible places and show that the resulting path is still covered by the upper path).
\end{proof}

Notice that the above proof does not show that
the interval in the statement of the proposition is the unique interval of rank 2 attaining the maximum of the M\"obius function.

Some computations also suggest the following conjecture:

\begin{conjecture}
The maximum absolute value of the M\"obius function on intervals of rank 3 is $(2n+1)\cdot n^2$, attained by the interval
$[U(UD)^{n-1}D,U(UD)^{n+2}D]$.
\end{conjecture}

Another intriguing conjecture, again supported by computational evidence, is the following:

\begin{conjecture}
The M\"obius function is alternating, meaning that it is $\geq 0$ on intervals of even rank and $\leq 0$ on intervals of odd rank.
\end{conjecture}

The above conjecture suggests that the intervals of the Dyck pattern poset might all be shellable.

\bigskip

Concerning the enumerative combinatorics of the intervals in the Dyck pattern poset, we still have to understand what happens in most of the cases.
Counting elements and covering relations in the case of initial intervals in which the maxiumum has exactly three peaks could be a good starting point.
Moreover,
there are some asymptotic issues that seem to be rather interesting.
Indeed, some computations suggest that the maximum size of an interval of fixed rank whose minimum has semilength $n$ is polynomial in $n$
(when $n$ tends to infinity).
Finally, we remark that this kind of investigations,
which has already been pursued for many combinatorially interesting posets (as we recalled in the Introduction),
seems to still be lacking for the permutation pattern poset.

\end{document}